\pdfoutput=1
\documentclass[12pt,reqno]{amsart}
\usepackage{amsfonts}
\usepackage[a4paper,inner=1in,outer=1in,vmargin=1in,marginparwidth=1in]{geometry}
\usepackage{amsmath,amsthm,amsfonts,amssymb}
\usepackage{bbm}
\usepackage{cite}
\usepackage{color}
\usepackage{enumerate}

\makeatletter
\@namedef{subjclassname@2020}{\textup{2020} Mathematics Subject Classification}
\makeatother

\def\C{\mathbb{C}}

\def\Z{\mathbb{Z}}

\newtheorem{theo}{Theorem}[section]

\newtheorem{lemm}[theo]{Lemma}
\newtheorem{remark}[theo]{Remark}
\newtheorem{coro}[theo]{Corollary}
\newtheorem{prop}[theo]{Proposition}

\numberwithin{equation}{section}
\allowdisplaybreaks
\begin{document}
\title{ Representations of Toroidal and Full toroidal Lie algebras over polynomial algebras}
 	
\author[S. Tantubay ]{Santanu Tantubay }
\address{S. Tantubay: Department of Mathematics, The Institute of Mathematical Sciences, A CI of Homi Bhabha National Institute, IV Cross Road, CIT Campus
Taramani
Chennai 600 113
Tamil Nadu, India}
\email{santanut@imsc.res.in, tantubaysantanu@gmail.com}
\author[P. Chakraborty]{Priyanshu Chakraborty}
\address{Priyanshu Chakraborty:School of Mathematical Sciences, Ministry of Education Key Laboratory of Mathematics and Engineering Applications and Shanghai Key Laboratory of PMMP,
		East China Normal University, No. 500 Dongchuan Rd., Shanghai 200241, China.}
\email{priyanshu@math.ecnu.edu.cn, priyanshuc437@gmail.com}

\thanks{$^\star$  Corresponding author}


\date{}

\maketitle
\begin{abstract}
Toroidal Lie algebras are $n$ variable generalizations of affine Kac-Moody Lie algebras. Full toroidal Lie algebra is the semidirect product of derived Lie algebra of toroidal Lie algebra and Witt algebra, also it can be thought of $n$-variable generalization of Affine-Virasoro algebras. Let $\tilde{\mathfrak{h}}$ be a Cartan subalgebra of a toroidal Lie algebra as well as full toroidal Lie algebra without containing the zero-degree central elements. In this paper, we classify the module structure on $U(\tilde{\mathfrak{h}})$ for all toroidal Lie algebras as well as full toroidal Lie algebras which are free $U(\tilde{\mathfrak{h}})$-modules of rank 1. These modules exist only for type $A_l (l\geq 1)$, $C_l (l\geq2)$ toroidal Lie algebras and the same is true for full toroidal Lie algebras. Also, we determined the irreducibility condition for these classes of modules for both the Lie algebras. \\

\end{abstract}

\section{Introduction}
In order to understand the module category of any algebras, one of the initial steps is to classify its irreducible objects. Lie algebra is a fundamental subject in mathematics. Lie algebra theory also plays an important role in physics. The classification of simple modules over finite dimensional simple complex Lie algebra seems wild; for $\mathfrak{sl}_2$ there is a weak version of such classification (see \cite{RB}, \cite{VM} ). Some other classes of simple modules for simple finite-dimensional Lie algebras are well studied. For example, simple weight modules with finite-dimensional weight spaces are classified for simple finite-dimensional Lie algebras, there are two types of such modules: parabolically induced modules and cuspidal module. Parabolically induced modules include simple finite-dimensional modules (\cite{EC}, \cite{JD}) and more generally highest-weight modules (\cite{JD}, \cite{JH}, \cite{BGG}). Simple cuspidal modules were classified by O. Matheieu (see \cite{OM}). Other well-studied classes of simple modules for finite-dimensional simple Lie algebras are Whittaker modules (\cite{BK}) and Gelfan-Zetlin modules (\cite{DOF}). \\
Affine Kac-Moody Lie algebras are natural generalizations of finite-dimensional simple Lie algebras. The structure theory and representation theory of affine Kac-Moody algebra has been used in different areas of mathematics (\cite{VK}, \cite{KR}, \cite{LL}) as well as in physics (\cite{BP}). Integrable weight representation theory is well-studied for affine Kac-Moody Lie algebras (See \cite{VK1}). Simple integrable modules with finite-dimensional weight spaces have been classified in \cite{VC} and \cite{CP}. Verma-type modules for this Lie algebras were first studied by Jakobsen-Kac (\cite{JK}) and then by Futorny (\cite{VF}, \cite{VF1}). The classification of all non-zero level simple weight modules with finite-dimensional weight spaces over Kac-Moody Lie algebras was obtained in \cite{FT} and for level zero case it was announced in \cite{DG}. There are some other classes of weight modules for affine Kac-Moody Lie algebras with infinite-dimensional weight spaces(see \cite{CP1}, \cite{BBFK}, \cite{FK}, \cite{MZ}) .\\
Non-weight simple modules for any Lie algebras are less understood. A class of non-weight modules for finite-dimensional simple Lie algebra of type $A_l$ is studied by J. Nilsson in \cite{N1}, which are known as Cartan-free modules. In \cite{N2}, the author showed that Cartan-free modules exist only for type $A_l(l\geq 1)$ and type $C_l(l\geq 2)$.  For affine Kac-Moody Lie algebras rank one Cartan free modules are classified in \cite{CTZ1}. For Witt algebra such classification was done in \cite{TZ}. For other classes of Lie algebras, classification of Cartan free modules one can see \cite{CH}, \cite{JZ}, \cite{XZ},.
This paper aims to study non-weight modules for two classes of Lie algebras namely toroidal Lie algebras and full toroidal Lie algebras. For integrable weight representations of toroidal Lie algebra and full toroidal Lie algebra, one can see \cite{ER}, \cite{RJ} and reference therein.\\
This paper is organized as follows: In Section \ref{NP}, we recall the definitions of Kac-Moody Lie algebra, toroidal Lie algebra, and full toroidal Lie algebra. Then we defined a full subcategory of the category of modules corresponding to a finite-dimensional commutative subalgebra of a Lie algebra. In Section \ref{sec: Aux Lem}, we recall some auxiliary lemmas related to finite-dimensional simple Lie algebras. In Section \ref{TA}, we classified all Cartan-free modules of rank one over toroidal Lie algebras and we determined the condition of simplicity for these modules. In Section \ref{FTA}, we classified all Cartan-free modules of rank one over full toroidal Lie algebras and we determined the condition of simplicity for these modules.

\section{Notations and Preliminaries}\label{NP}
Throughout this paper, $\Z$, $\C$, and $\C^*$ denote the sets of integers, complex numbers, and nonzero complex numbers, respectively.  For a Lie algebra $\mathfrak{a}$, we denote the universal enveloping algebra of $\mathfrak{a}$ as $\mathcal{U}(\mathfrak{a})$. All the vector spaces, algebras, and tensor products are over $\C$, unless it is specified.

In this section, we shall introduce the notions of affine Kac-Moody algebras, toroidal Lie algebra and full toroidal Lie algebra.  
\subsection{}
Suppose $\mathfrak{g}=\mathfrak{g}(X_l)$ is a  finite-dimensional simple Lie algebra with a fixed Cartan subalgebra $\mathfrak{h}$ of type $X_l$, where $X$ is one of the types $A,B,\cdots G$ . Let $\{x_i, y_i, h_i: 1\leq i\leq l\}$ be a set of Chevalley generators for $\mathfrak{g}$, where $l=\hbox{dim} (\mathfrak{h})$. For type $\mathfrak{g}=\mathfrak{g}(A_l)$ or type $\mathfrak{g}=\mathfrak{g}(C_l)$, we take a basis $\{H_i:1\leq i\leq l\}$ of $\mathfrak{h}$ by 
$$[H_i, x_j ]=\delta_{ij}x_i,\; [H_i, y_j]=-\delta_{ij}y_j,\; \forall \; 1\leq i,j\leq l.$$

\subsection{Affine Kac-Moody algebras.}\label{Aff}
Let $\mathfrak{g}$ be a simple finite-dimensional Lie algebra. Let us fix a Cartan subalgebra $\mathfrak{h}$ of $\mathfrak{g}$. Suppose that $\Delta$ is the set of roots of $\mathfrak{g}$ and $\Delta_+$ is the set of positive roots of $\mathfrak{g}$. Then $\mathfrak{g}=\mathfrak{h}\oplus \displaystyle{\bigoplus_{\alpha\in \Delta}} \mathfrak{g}_\alpha $ is the root space decomposition of $\mathfrak{g}$ with respect to $\mathfrak{h}$. Assume that $(-,-)$ is the Killing form on $\mathfrak{g}$.

Let $\C[t^{\pm1}]$ be the Laurent polynomial ring over $\C$. Then the algebra $ \widehat{\mathfrak{g}}$ is a central extension of loop algebra $\mathfrak{g}\otimes \C [t^{\pm1}]$, i.e. $\widehat{\mathfrak{g}}=\mathfrak{g}\otimes \C [t^{\pm1 }]\oplus \C {\bf{k}}$ with the Lie brackets
\begin{equation}\label{affLA}
[x(m),y(n)]=[x,y](m+n)+m(x,y)\delta_{m+n,0}{\bf{k}}, \quad [\widehat{\mathfrak{g}},{\bf{k}}]=0,
\end{equation}
for all $x,y\in \mathfrak{g}, m,n\in \Z$, where $x(m)=x\otimes t^m$.

Moreover, we can define the affine Kac-Moody algebra  $\widetilde{\mathfrak{g}}:= \widehat{\mathfrak{g}} \oplus \C d = \mathfrak{g}\otimes \C [t^{\pm1 }]\oplus \C {\bf{k}} \oplus \C d$, with above Lie brackets and  $$[d,x(m)]= mx(m).$$
Let $\widetilde{\mathfrak{h}}=\mathfrak{h}\oplus \C d$ be a Cartan subalgebra of Kac-Moody Lie algebra $\widetilde{\mathfrak{g}}$.
\subsection{Toroidal Lie algebras}
Let $A=\C[t_1^{\pm1}, \cdots, t_n^{\pm1}]$ be the Laurent polynomial ring in $n$ commuting variables $t_1, \cdots, t_n$ over $\C$. Assume $\Omega_A$ is the free $A$-module with basis ${K_1,\dots,K_n}$ and $dA$ is the subspace spanned by the elements $\sum_{i=1}^nm_it^mK_i,\; m\in \Z ^n$. 
Then $L(\mathfrak{g})=\mathfrak{g} \otimes A$ with the brackets $[x\otimes t^r, y\otimes t^s]=[x,y] \otimes t^{r+s}$ is known as the loop algebra of $\mathfrak{g}$ by $A$. Define $\mathcal{K}_A=\Omega_A/dA$, then it is well known that (See \cite{RMY}, \cite{K}) $\widetilde{L(\mathfrak{g})}=L(\mathfrak{g}) \oplus \mathcal{K}_A$ is the universal central extension of $L(\mathfrak{g})$ with the Lie brackets
\begin{equation}\label{tla}
    [x(a), y(b)]=[x,y](a+b)+ (x,y)\sum_{i=1}^na_1t^{a+b}K_i, \; [L(\mathfrak{g}), \mathcal{K}_A]=0
\end{equation}

Let $D=\oplus_{i=1}^n \C d_i$ be a $n$ dimensional vectorspace. We define the toroidal Lie algebra $\widetilde{\mathfrak{g}}_n=\widetilde{L(\mathfrak{g})} \oplus D $ with the Lie brackets (\ref{tla}) and $[d_i,x(a)]=a_ix(a),\; [d_i,d_j]=0$, where $x(a) \in \widetilde{L(\mathfrak{g})}$ and $1\leq i,j\leq n$. Then it is easy to see that $\widetilde{\mathfrak{g}}_1=\widetilde{\mathfrak{g}}$.
We define a Cartan subalgebra of toroidal Lie algebra by $\widetilde{\mathfrak{h}}_n=\mathfrak{h}\oplus D$.

\subsection{Full toroidal Lie algebras (\cite{RJ})}
Let $Der(A)$ be the space of all derivations on $A$. It has a basis consisting of elements of the form $\{t^rd_i: r\in \Z^n,\; 1\leq i\leq n\}$. Now define $D(u,r)=\sum_{i=1}^nu_it^rd_i$, where $u\in \C^n$ and $r\in \Z^n$. We see that $Der(A)$ is a Lie algebra, with the Lie brackets $[D(u,r), D(v,s)]=D(w,r+s)$, where $w=(u,s)v-(v,r)u$, $u,v\in \C^n,\; r,s\in \Z^n$. This Lie algebra is known as Witt algebra. Every derivation $d$ of $A$ can be naturally extended to a derivation on the tensor product $\mathfrak{g}\otimes A$ by $d(x\otimes f)=x\otimes df$, where $x\in \mathfrak{g}$, $f\in A$ and $d$ has a unique natural extension to $\mathcal{K}_A$ by 
\[t^rd_i(t^sK_j)=s_it^{(r+s)}K_j+\delta_{ij}\sum_{p=1}^nr_pt^{r+s}K_p. \]
It is well known that $Der(A)$ admits two non-trivial $2$-cocyles with values in $\mathcal{K}_A$ (See \cite{BB}):
\[\phi_1(t^rd_i,t^sd_j)=-s_ir_j\sum_{p=1}^nr_pt^{r+s}K_p,\]
\[\phi_2(t^rd_i,t^sd_j)=r_is_j\sum_{p=1}^nr_pt^{r+s}K_p.\]
Let $\phi$ be an arbitrary Linear combination of $\phi_1$ and $\phi_2$. Then there is a corresponding Lie algebra
\[\mathcal{L}=\mathfrak{g}\otimes A\oplus \mathcal{K}_A\oplus Der(A)\]
with the Lie brackets \ref{tla} and the following:
\[[t^rd_i,t^sK_j]=s_it^{r+s}K_j+\delta_{ij}\sum_{p=1}^nr_pt^{r+s}K_p\]
\[[D(u,r),D(v,s)]=D(w,r+s)+\phi(D(u,r),D(w,s))\]
\[[D(u,r), x(s)]=(u,s)x(r+s).\]
This lie algebra is known as full toroidal Lie algebra.
\subsection{} 
Let $\mathfrak{a}$ be any Lie algebra over $\C$ and $\mathfrak{b}$ be any finite-dimensional abelian subalgebra of $\mathfrak{a}$. Denote $\mathcal{M}(\mathfrak{a},\mathfrak{b})$ by the full subcategory of $U(\mathfrak{a})$-modules consisting of objects whose restriction to $U(\mathfrak{b})$ is a free module of rank $1$, i.e.,
$$\mathcal{M}(\mathfrak{a}, \mathfrak{b})=\{M \in U(\mathfrak{a})-\hbox{Mod}| \hbox{Res}_{U(\mathfrak{b})}^{U(\mathfrak{a})}M \cong_{U(\mathfrak{b})}  U(\mathfrak{b})\}.$$

\section{Some Auxiliary Lemmas}\label{sec: Aux Lem}
In this section, we will recall several known results which are essential for our paper.
\begin{lemm}\label{fdss}{(\cite{N2}, Cor.11)}
Let $\mathfrak{g}$ be a finite-dimensional simple Lie algebra with a Cartan subalgebra $\mathfrak{h}$. Then $\mathcal{M}(\mathfrak{g}, \mathfrak{h})\neq \varnothing$ if and only if $\mathfrak{g}$ is of type $A_l(l\geq 1)$ or $C_l(l\geq 2)$.
\end{lemm}
For any finite-dimensional abelian Lie algebra $Z$, set $\tilde{\mathbf{g}}=\mathfrak{g}\oplus Z$ be the trivial central extension. Set $\tilde{\mathfrak{h}}=\mathfrak{h}\oplus Z$, this will be again abelian finite-dimensional Lie algebra.
\begin{lemm}\label{fdssae}
   Let $\mathfrak{g}$ be a finite-dimensional simple Lie algebra and ${Z}$ be a finite-dimensional Lie algebra.
   \begin{enumerate}
       \item Suppose $\mathfrak{g}$ is of type $A_l(l\geq 1)$. For any $\mathbf{a}=(a_1,\cdots, a_l)\in (\C^*)^l$, $b\in \mathcal{U}(Z)$, $\subseteq \overline{1,l+1}$, the vector space $M(\mathbf{a},b,S)=\mathcal{U}(\Tilde{\mathfrak{h}})$ becomes a $\tilde{\mathfrak{g}}$-module under the following action: for all $g\in \mathcal{U}(\Tilde{\mathfrak{h}})$,
        
       $$H_i. \tilde{g}=H_i\tilde{g}, \;i\in \mathbf{l}; \hspace{1cm}z.\tilde{g}=z\tilde{g},\; \forall z \in Z, $$
       
       $$x_i.\tilde{g}= a_i\{\delta_{i\in S}+ \delta_{i\notin S}(H_i-H_{i-1}-b-1)\}\{\delta_{i+1\in S} (H_{i+1}-H_i-b)+\delta_{i+1\notin S}\}w_i(\tilde{g}),$$

      $$y_i.\tilde{g}= a_i^{-1}\{\delta_{i\in S}(H_i-H_{i-1}-b)+ \delta_{i\notin S} \}\{\delta_{i+1\in S}  +\delta_{i+1\notin S}(H_{i+1}-H_i-b-1)\}w_i^{-1}(\tilde{g}),$$
      
      where $H_{l+1}=H_0:=0$ and $w_i$ is the automorphism of $\mathcal{U}(\tilde{\mathfrak{h}})$ which sends $H_i$ to $H_{i}-1$ and fixes other $H_j, \; j\neq i$ and $Z$. Moreover,
      $$\mathcal{M}(\tilde{\mathfrak{g}}, \tilde{\mathfrak{h}})=\{M(a,b,S): a\in (\C^*)^l,\; b\in \mathcal{U}(Z), S\subseteq \overline{1,l+1}\}.$$

      \item  $\mathcal{M}(\tilde{\mathfrak{g}}, \tilde{\mathfrak{h}})=\empty$ if and only if $\mathcal{M}(\mathfrak{g}, \mathfrak{h})=\varnothing$.
   \end{enumerate}
\end{lemm}

Now suppose $\mathfrak{g}$ is of type $C_l(l\geq 2)$. Now for any $\mathbf{a}\in (\C^*)^l, \; S\subseteq \overline{1,l}$, let us recall the $\tilde{\mathfrak{g}}$-module $M(\mathbf{a},S)=\mathcal{U}(\tilde{\mathfrak{h}})$ from \cite{CTZ2} with the following actions:
$$H_i. \tilde{g}=H_i\tilde{g}, \;i\in \mathbf{l}; \hspace{1cm}z.\tilde{g}=z\tilde{g},\; \forall z \in Z,$$
$$x_k.\tilde{g}= a_k\{\delta_{k\in S}+ \delta_{k\notin S}(H_k-H_{k-1}-b-\frac{1}{2})\}\{\delta_{k+1\in S} ((1+\delta_{k,l-1})H_{k+1}-H_k+\frac{1}{2})+\delta_{k+1\notin S}\}w_k(\tilde{g}),$$

$$y_k.\tilde{g}= a_k^{-1}\{\delta_{k\in S}(H_k-H_{k-1}-\frac{1}{2})+ \delta_{k\notin S} \}\{\delta_{k+1\in S}  +\delta_{k+1\notin S}((1+\delta_{k,l-1})H_{k+1}-H_k-\frac{1}{2})\}w_k^{-1}(\tilde{g}),$$

      $$x_l.\tilde{g}= a_l\{\delta_{l\in S}- \delta_{l\notin S}(H_l-\frac{1}{2}H_{l-1}-\frac{3}{4})(H_l-\frac{1}{2})H_{l-1}-\frac{1}{4}\}w_l(\tilde{g}),$$

      $$ y_l.\tilde{g}= a_l^{-1}\{-\delta_{l\in S}(H_l-\frac{1}{2}H_{l-1}+\frac{3}{4})(H_l-\frac{1}{2}H_{l-1}+\frac{1}{4})+ \delta_{l\notin S} \}w_l^{-1}(\tilde{g}),$$
      \vspace{.5cm}
 where $\tilde{g}\in \mathcal{U}(\tilde{\mathfrak{h}}),\; 1\leq k\leq l-1$, $H_0:=0$ and $w_i$ is the automorphism of $\mathcal{U}(\tilde{\mathfrak{h}})$ which sends $H_i$ to $H_{i}-1$ and fixes other $H_j, \; j\neq i$ and $Z$. 

 \begin{lemm}{(\cite{CTZ1}, Lemma 4)}
     Let $\mathfrak{g}=\mathfrak{sp}_{2l}(l\geq 2)$ with a Cartan subalgebra $\mathfrak{h}$. For a finite-dimensional abelian Lie algebra $Z$, we have $\mathcal{M}(\tilde{\mathfrak{g}}, \tilde{\mathfrak{h}})=\{M(\mathbf{a},S): \mathbf{a}\in (\C^*)^l,\; S\subseteq \overline{1,l}\}$. 
 \end{lemm}

 \begin{lemm}{(\cite{CTZ1}, \cite{N1}, \cite{N2})}\label{simplecondition}
     \begin{enumerate}
         \item $M(\mathbf{a},b, S)\in \mathcal{M}(\mathfrak{g}(A_l), \mathfrak{h})$ is simple if and only if $(l+1)b \notin \Z_+$ or $S\neq \emptyset$ or $S\neq \overline{1,l+1}$
         \item Any module in $\mathcal{M}(\mathfrak{g}(C_l), \mathfrak{h})$ is simple.
     \end{enumerate}
 \end{lemm}
 \begin{lemm}\label{pa}{(\cite{CTZ2}), Lemma 2)}
   Let $\C[T_1, \dots, T_m]$ be a polynomial algebra in $m$ indeterminates $T_1, \dots, T_m$ over $\C$ and $w_i$ be the automorphism of $\C[T_1, \dots, T_m]$ which sends $T_i$ to $T_i-1$ and fixes $T_j$ for $j\neq i$. For any $g\in \C[T_1, \dots, T_m],\; k\in \Z\setminus \{0\}$ and $k^\prime\in \Z_+$ we have
   $$\hbox{deg}_{T_i}(w_i^k-\hbox{Id})(g)=\hbox{deg}_{T_i}(g)-1,$$
   and 
   $$\hbox{deg}_{T_i}(w_i-\hbox{Id})^{k^\prime}(g)=\begin{cases}  \hbox{deg}_{T_i}(g)-k^\prime &
\hbox{if}\; k^\prime \leq deg_{T_i}(g) \\
 -1 &\hbox{if}\; k^\prime > deg_{T_i}(g) ,\end{cases} $$
 where $\hbox{deg}_{T_i}(0)=-1$
 \end{lemm}
\section{Cartan Free modules over $\widetilde{\mathfrak{g}}_n$}\label{TA}
In this section, we will study the categories $\mathcal{M}(\widetilde{\mathfrak{g}}_n, \widetilde{\mathfrak{h}}_n)$ for toroidal Lie algebras $\widetilde{\mathfrak{g}}_n$.\\

We see that $\widetilde{\mathfrak{h}}_n$ is an abelian Lie algebra, so $\mathcal{U}(\widetilde{\mathfrak{h}}_n)=\C[H_1, \dots, H_l, d_1, \dots, d_n]$. Now define some automorphisms of $\C[H_1, \dots, H_l, d_1, \dots, d_n]$ as follows: for any $\tilde{g}(H_1, \dots, H_l, d_1, \dots, d_n)\in  \mathcal{U}(\widetilde{\mathfrak{h}}_n),\; 1\leq i\leq l,\; 1\leq j\leq n$, 
$$\sigma_i: \mathcal{U}(\widetilde{\mathfrak{h}}_n)\rightarrow \mathcal{U}(\widetilde{\mathfrak{h}}_n),\; \tilde{g}(H_1, \dots, H_l, d_1, \dots, d_n) \mapsto \tilde{g}(H_1, \dots,H_i-1, \dots, H_l, d_1, \dots, d_n);$$
$$\tau_j:\mathcal{U}(\widetilde{\mathfrak{h}}_n)\rightarrow \mathcal{U}(\widetilde{\mathfrak{h}}_n),\; \tilde{g}(H_1, \dots, H_l, d_1, \dots, d_n) \mapsto \tilde{g}(H_1, \dots, H_l, d_1, \dots,d_j-1,\dots,  d_n).$$
Note that $\{\sigma_i,\tau_j:1\leq i\leq l, \; 1\leq j\leq n\}$ is a commuting family of automorphisms of $\mathcal{U}(\widetilde{\mathfrak{h}}_n)$.\\
Now we define $\mathcal{P}_H:= \C[H_1, \dots, H_l],\; \mathcal{P}_D:= \C[ d_1, \dots, d_n]$ and\\ $\mathcal{P}_i=\C[H_1, \dots,H_{i-1},H_{i+1},\dots  H_l, d_1, \dots, d_n],\; \mathcal{P}^j=\C[H_1, \dots, H_l, d_1, \dots,d_{j-1},d_{j+1},\dots d_n],$ for $1\leq i\leq l,\; 1\leq j \leq n$.
Let $\mathbf{a}=(a_1)\in \Z^n$ and for any $\mathbf{\lambda}=(\lambda_1,\dots ,\lambda_n)\in (\C^*)^n$ we define $\mathbf{\lambda}^{\mathbf{a}}=\lambda_1^{a_1}\dots \lambda_n^{a_n}\in \C^*$ and $\tau^{\mathbf{a}} =\tau_1^{a_1}\dots \tau_n^{a_n}\in \hbox{Aut}(\mathcal{U}(\widetilde{\mathfrak{h}}_n))$. 

For any $M_\mathfrak{g}\in \mathcal{M}(\mathfrak{g},\mathfrak{h})$, we have $M_\mathfrak{g}=\mathcal{U}(\mathfrak{h})\subset \mathcal{U}(\widetilde{\mathfrak{h}}_n)$. Let $\mathbf{\lambda}\in (\C^*)^n$. For any $\tilde{g}\in \mathcal{U}(\widetilde{\mathfrak{h}}_n)$ and $a\in \Z^n$, we define
\begin{equation}\label{ms}
\begin{cases}
    t^aK_j. \tilde{g}=0,\; d_j.  \tilde{g}=d_j\tilde{g}\\
    h(a).  \tilde{g}=\mathbf{\lambda}^ah\tau^a( \tilde{g})\\
    x_i(a).  \tilde{g}= \mathbf{\lambda}^a\sigma_i\tau^a(\tilde{g})(x_i.1)\\
     y_i(a).  \tilde{g}= \mathbf{\lambda}^a\sigma_i^{-1}\tau^a(\tilde{g})(y_i.1)
\end{cases}
\end{equation}
where $i\in \overline{1,l}$ and $j\in \overline{1,n}$ and $x_i.1, y_i.1$ are determined by the $\mathfrak{g}$-module structure on $M_{\mathfrak{g}}$. It is easy to see that (\ref{ms}) defines a $\mathcal{U}(\widetilde{\mathfrak{g}}_n)$-module structure on $\mathcal{U}(\widetilde{\mathfrak{h}}_n)$. We denote the corresponding module as $\widetilde{M}_{\widetilde{\mathfrak{g}}_n}(\mathbf{\lambda},M_\mathfrak{g})$.
It is easy to see that $\widetilde{M}_{\widetilde{\mathfrak{g}}_n}(\mathbf{\lambda},M_\mathfrak{g}) \in \mathcal{M}(\widetilde{\mathfrak{g}}_n, \widetilde{\mathfrak{h}}_n)$. We see that as vector space $\widetilde{M}_{\widetilde{\mathfrak{g}}_n}(\mathbf{\lambda},M_\mathfrak{g})\cong M_\mathfrak{g} \otimes \mathcal{P}_D$. We can also define a  $\widetilde{\mathfrak{g}}_n$-module structure on $M_\mathfrak{g} \otimes \mathcal{P}_D$ by similar actions as in \ref{ms}. In this case we will have $\widetilde{M}_{\widetilde{\mathfrak{g}}_n}(\mathbf{\lambda},M_\mathfrak{g})\cong M_\mathfrak{g} \otimes \mathcal{P}_D$ as $\widetilde{\mathfrak{g}}_n$-module.
\begin{prop}\label{prop charc}
    The $\widetilde{\mathfrak{g}}_n$-module $\widetilde{M}_{\widetilde{\mathfrak{g}}_n}(\mathbf{\lambda},M_\mathfrak{g})$ is simple if and only if the corresponding $\mathfrak{g}$-module $M_\mathfrak{g}$ is simple.
\end{prop}
\begin{proof}
Suppose $M_{\mathfrak{g}}$ is not simple as $\mathfrak{g}$-module and $N$ be a non-zero proper submodule of $M_{\mathfrak{g}}$. Then by (\ref{ms}) we see that $N\otimes \mathcal{P}_D$ is a nonzero proper 
$\widetilde{\mathfrak{g}}_n$-module of $\widetilde{M}_{\widetilde{\mathfrak{g}}_n}(\mathbf{\lambda},M_\mathfrak{g})$.    \\
Conversely assume $M_{\mathfrak{g}}$ is a simple $\mathfrak{g}$-module. Let $W$ be a nonzero submodule of $M_{\mathfrak{g}}\otimes \mathcal{P}_D$.\\
{\bf Claim:} There exists $w\in W$ such that $\hbox{deg}_{d_j}(w)=0$ for all $1\leq j\leq n$.\\
Let $w_1\in W$ be a non-zero element with a minimal degree in $d_1$. Now for any $h\in \mathfrak{h}$, we see that $\hbox{deg}_{d_1}(h(e_1)-\lambda_1h)(w_1)<\hbox{deg}_{d_1}(w_1)$, hence $\hbox{deg}_{d_1}(w_1)=0$. Now let $w_2\in W$ be a nonzero element from the set $\{w\in W: \hbox{deg}_{d_1}(w)=0\}$ with minimal degree in $d_2$. Now again for any $h\in \mathfrak{h}$, we have $\hbox{deg}_{d_2}(h(e_2)-\lambda_2h)(w_2)<\hbox{deg}_{d_2}(w_2)$, which forces $\hbox{deg}_{d_2}(w_2)=0$. Continuing this process we find $w\in W$ such that $\hbox{deg}_{d_j}(w)=0$ for all $1\leq j\leq n$.\\
Hence $w\in M_{\mathfrak{g}}$, which implies $\mathcal{U}(\mathfrak{g})w=M_{\mathfrak{g}}$. Note that $ M_{\mathfrak{g}} \otimes \mathcal{P}_D\subset W$, we obtain that $W=M_{\mathfrak{g}} \otimes \mathcal{P}_D$. So the $\widetilde{\mathfrak{g}}_n$-module $\widetilde{M}_{\widetilde{\mathfrak{g}}_n}(\mathbf{\lambda},M_\mathfrak{g})$ is simple.
\end{proof}
From this Proposition \ref{prop charc} and Lemma \ref{simplecondition}, we have the following Corollary. 
\begin{coro}\label{coro}
 \begin{enumerate}
    \item  The $\widetilde{\mathfrak{g}}_n(A_l)$-module $\widetilde{M}_{\widetilde{\mathfrak{g}}_n(A_l)}(\mathbf{\lambda},M(\mathbf{a},b, S))$ is simple if and only if $(l+1)b\notin \Z_+$ or $1\leq |S|\leq l.$
     \item  $\widetilde{\mathfrak{g}}_n(C_l)$-module $\widetilde{M}_{\widetilde{\mathfrak{g}}_n(C_l)}(\mathbf{\lambda},M_{\mathfrak{g}(C_l)}$ is simple for any $\mathbf{\lambda}\in (\C^*)^n$ and $M_{\mathfrak{g}(C_l)}\in \mathcal{M}(\mathfrak{g}(C_l),\mathfrak{h})$.
\end{enumerate}    
\end{coro}
     
\begin{remark}
If $M\in \mathcal{M}(\widetilde{\mathfrak{g}}_n, \widetilde{\mathfrak{h}}_n)$, then $\hbox{Res}_{\tilde{\mathfrak{g}}}^{\widetilde{\mathfrak{g}}_n}(M) \in \mathcal{M}(\tilde{\mathfrak{g}},\tilde{\mathfrak{h}})$, where $\tilde{\mathfrak{g}}=\mathfrak{g}\oplus D$ and  $\tilde{\mathfrak{h}}=\mathfrak{h}\oplus D$. Then from Lemma \ref{fdssae}, we have $\mathcal{M}(\widetilde{\mathfrak{g}}_n, \widetilde{\mathfrak{h}}_n) \neq\varnothing$ only for type $A_l(l\geq 1)$ and $C_l(l\geq 2)$.
\end{remark} 
\begin{theo}\label{main1}
 Let $\mathfrak{g}$ be a finite-dimensional Lie algebra of type $A_l(l\geq 1)$ or of type $C_l(l\geq 2)$, then 
 $$\mathcal{M}(\widetilde{\mathfrak{g}}_n, \widetilde{\mathfrak{h}}_n)=\{\widetilde{M}_{\widetilde{\mathfrak{g}}_n}(\mathbf{\lambda},M_\mathfrak{g}): \mathbf{\lambda}\in (\C^*)^n, M_{\mathfrak{g}}\in \mathcal{M}(\mathfrak{g}, \mathfrak{h})\}.$$
\end{theo} 
To prove this theorem it is enough to prove that for any $\widetilde{M}_{\widetilde{\mathfrak{g}}_n}\in  \mathcal{M}(\widetilde{\mathfrak{g}}_n, \widetilde{\mathfrak{h}}_n)$ there exists a pair $(\mathbf{\lambda}, M_{\mathfrak{g}})\in (\C^*)^n\times \mathcal{M}(\mathfrak{g},\mathfrak{h})$, such that $\widetilde{M}_{\widetilde{\mathfrak{g}}_n}=\widetilde{M}_{\widetilde{\mathfrak{g}}_n}(\mathbf{\lambda}, M_{\mathfrak{g}})$.\\
We have the following easy observations.
\begin{lemm}\label{eva}
    Suppose $\tilde{g}\in \widetilde{M}_{\widetilde{\mathfrak{g}}_n},\; a \in \Z^n$, $i\in \overline{1,l},\;  j\in \overline{1,n}$. Then we have 
    \begin{equation*}
    \begin{cases}
        t^aK_j. (\tilde{g})=\tau^a (\tilde{g})(t^aK_j.1)\\
        h(a). (\tilde{g})=\tau^a (\tilde{g})(h(a).1)\\
        x_i(a).(\tilde{g})=(\sigma_i \tau^a(\tilde{g})) (x_i(a).1)\\
        y_i(a).(\tilde{g})=(\sigma_i^{-1} \tau^a(\tilde{g})) (y_i(a).1)
    \end{cases}
     \end{equation*}
     \end{lemm}
 Lemma \ref{eva} shows that to prove  $\widetilde{M}_{\widetilde{\mathfrak{g}}_n}=\widetilde{M}_{\widetilde{\mathfrak{g}}_n}(\mathbf{\lambda},M_{\mathfrak{g}})$ for some $\mathbf{\lambda}\in (\C^*)^n$ and $M_{\mathfrak{g}}\in \mathcal{M}(\mathfrak{g}, \mathfrak{h})$ it is enough to prove the following two assertions:\\
 {\bf Assertion (A):} $t^aK_j.1=0$ for $j\in \overline{1,n}$.\\
 {\bf Assertion (B):} $X(a).1=\mathbf{\lambda}^a(X.1)$ for some $\mathbf{\lambda}\in (\C^*)^n$ and $X.1$ is determined by some $\mathcal{U}(\mathfrak{g})$-module $M_{\mathfrak{g}}\in \mathcal{M}(\mathfrak{g}, \mathfrak{h})$, where $X=x_i,y_i, h_i$ for all $i\in \overline{1,l}$.\\
 We will prove these two assertions for type $A_l(l\geq 1)$ and the same method is valid for type $C_l(l\geq 2)$.
 \begin{lemm}\label{lemm4.6}
     For $\widetilde{M}_{\widetilde{\mathfrak{g}}_n}$, we have
     \begin{enumerate}
         \item $x_i(e_j).1\neq 0,\; y_i(e_j).1 \neq 0$ for all $i\in \overline{1,l}$ and $j\in \overline{1,n}$.
         \item There exists $a_j\in \C$ such that $K_j.1=a_j$.
         \end{enumerate}
 \end{lemm}
 \begin{proof}
             The proof is similar to Lemma 12 of \cite{CTZ1}. Just note that we need to run the same Lemma $n$ times. 
         \end{proof}
         \begin{lemm}\label{lemm4.7}
             \begin{enumerate}
                 \item In (\ref{fdssae}), $b\in \C$, hence $x_i.1, y_i.1\in \mathcal{U}(\mathfrak{h})$.
                 \item $K_j.1=0$, for all $j\in \overline{1,n}$.
                 \item There exists $\lambda_i^j\in \C^*$ such that 
                 $$x_i(-e_j).1=\frac{1}{\lambda_i^j}(x_i.1), \; y_i(e_j).1=\lambda_i^j(f_i.1).$$
             \end{enumerate}
         \end{lemm}
         \begin{proof}
             The proof is similar to Lemma 13 of \cite{CTZ1}.
         \end{proof}
         \begin{lemm}\label{lemm4.8}
         \begin{enumerate}
                 \item $h_i(\pm e_j).1 =(\lambda_i^j)^{\pm 1}h_i.1.$
                 \item $\lambda_1^j=\lambda_2^j=\cdots=\lambda_l^j$ for all $j\in \overline{1,n}$.    
                 \item For any $a\in \Z^n$ and $x\in \mathfrak{g}$,we have $x(a)=\mathbf{\lambda}^a(x.1)$ for some $\mathbf{\lambda}\in(\C^*)^n$.
             \end{enumerate}
             \end{lemm}
             \begin{proof}
          From Lemma 14 of \cite{CTZ1}, we can prove the first two parts of the proposition together with $x(re_j)=(\lambda^j)^r(x.1)$, where $\lambda^j=\lambda_i^j\in \C^*$, $r\in \Z$ and $x\in \mathfrak{g}$. \\
          Now we have $x_i(r_je_j+r_ke_k)=[H_i(r_je_j),x_i(r_ke_k)]$,  where $1\leq i\leq l$ and $1\leq j\neq k \leq n$, therefore 
        \begin{align*}
        x_i(r_je_j+r_ke_k).1
        &=H_i(r_je_j).x_i(r_ke_k).1-x_i(r_ke_k).H_i(r_je_j).1 \cr
        &=\lambda_{k}^{r_k}H_i(r_je_j)(x_i.1)-\lambda_{j}^{r_j}x_i(r_ke_k)(H_i.1) \cr
        &=\lambda_k^{r_k}(\tau_j^{r_j}(x_i.1))H_i(r_je_j).1-\lambda_{j}^{r_j}(\sigma_i\tau_k^{r_k}(H_i.1))(x_i(r_ke_k).1) \cr
        &=\lambda_i^{r_j}\lambda_k^{r_k}x_i.1
         \end{align*}     
         Similarly from the Lie brackets $y_i(r_je_j+r_ke_k)=[y_i(r_je_j),H_i(r_ke_k)]$, where $1\leq i\leq l$ and $1\leq j\neq k \leq n$, we prove $y_i(r_je_j+r_ke_k)=\lambda_i^{r_j}\lambda_k^{r_k}y_i.1$
              
             Proceeding inductively we can prove $x_i(r).1=\lambda^rx_i.1$ and $y_i(r).1=\lambda^ry_i.1$ for any $r\in \Z^n$. \\
             Now 
             \begin{align*}
                  H_i(r).1 &=[x_i(0),y_i(r)].1 \cr
                  &=x_i(0)y_i(r).1-y_i(r)x_i(0).1 \cr
                  &=\lambda^{r}x_i.y_i.1- \sigma_i^{-1}\tau^r(x_i.1)y_i(r).1 \cr
                  &=\lambda^r\{x_i.y_i.1-\sigma_i^{-1}(x_i.1)y_i.1\} \cr
                  &=\lambda^r\{x_i.y_i.1-y_i.x_i.1\} \cr
                  &=\lambda^rH_i.1,
             \end{align*}
             hence we have part 3 of our Lemma.
\end{proof}
\begin{lemm}\label{lemm4.9}
    $t^aK_j.1=0$ for $j\in \overline{1,n}$ and $0\neq a \in \Z^n$.
\end{lemm}
 \begin{proof}
 We have $t^aK_j.1=[x_i(e_j),y_i(a-e_j)].1-H_i(a).1=x_i(e_j).y_i(a-e_j).1-y_i(a-e_j).x_i(e_j).1-\lambda^aH_i.1=\lambda^{(a-e_j)}x_i(e_j)(y_i.1)-\lambda_jy_i(a-e_j).(x_i.1)-\lambda^aH_i=\lambda^{(a-e_j)} \sigma_i\tau_j(y_i.1)(x_i(e_j).1)-\lambda_j\sigma_i^{-1}\tau^{(a-e_j)}(x_i.1)(y_i(a-e_j).1)-\lambda^a H_i=\lambda^a\sigma_i(y_i.1)(x_i.1)-\lambda^a\sigma_i^{-1}(x_i.1)(y_i.1)-\lambda^aH_i=0$.   
 \end{proof}            

  The proof of Theorem \ref{main1} will follow from, Lemma \ref{eva}, Lemma \ref{lemm4.6}, Lemma \ref{lemm4.7}, Lemma \ref{lemm4.8} and Lemma \ref{lemm4.9}.          
             
\section{Cartan free modules over full toroidal Lie algebras}\label{FTA}
  
\subsection{}
In this section, we will consider the categories $\mathcal{M}(\mathcal{L}, \widetilde{\mathfrak{h}}_n)$. First, we construct a class of modules in this category. Suppose $M_{\mathfrak{g}}\in \mathcal{M}(\mathfrak{g},\mathfrak{h}), \; \Lambda\in (\C^*)^n$ and $a\in \C$. For any $g\in \mathcal{U}(\widetilde{\mathfrak{h}}_n)$ and $r\in\Z^n$, we define the following action:
\begin{equation}\label{ms1}
\begin{cases}
    t^rK_j. \tilde{g}=0, \\
    h(a).  \tilde{g}=\mathbf{\lambda}^ah\tau^a( \tilde{g})\\
    x_i(a).  \tilde{g}= \mathbf{\lambda}^a\sigma_i\tau^a(\tilde{g})(x_i.1)\\
     y_i(a).  \tilde{g}= \mathbf{\lambda}^a\sigma_i^{-1}\tau^a(\tilde{g})(y_i.1)\\
     t^rd_i.\tilde{g}=\lambda^r \tau^r(\tilde{g})(d_i-r_i(a+1)).
\end{cases}
\end{equation}
We see that with the above actions $\mathcal{U}(\widetilde{\mathfrak{h}}_n)$ becomes $\mathcal{L}$-module where $\widetilde{\mathfrak{h}}_n$ acts freely on $\mathcal{U}(\widetilde{\mathfrak{h}}_n)$. We denote this module by $\widehat{M}_{\mathcal{L}}(\Lambda,a,M_{\mathfrak{g}})$.

\begin{prop}
    The $\mathcal{L}$-module $\widehat{M}_{\mathcal{L}}(\Lambda,a,M_{\mathfrak{g}})$ is simple if and only if the corresponding $\mathfrak{g}$-module $M_{\mathfrak{g}}$ is simple.
\end{prop}
\begin{proof}
    The proof is same as Proposition \ref{prop charc}.
\end{proof}
\begin{coro}\label{coro1}
 \begin{enumerate}
    \item  The $ \mathcal{L}(A_l)$-module $\widehat{M}_{\tau(A_l)}(\mathbf{\lambda},a,M(\mathbf{b},c, S))$ is simple if and only if $(l+1)c\notin \Z_+$ or $1\leq |S|\leq l.$
     \item  $ \mathcal{L}(C_l)$-module $\widehat{M}_{\tau(C_l)}(\mathbf{\lambda},a,M_{\mathfrak{g}(C_l)})$ is simple for any $\mathbf{\lambda}\in (\C^*)^n$ and $M_{\mathfrak{g}(C_l)}\in \mathcal{M}(\mathfrak{g}(C_l),\mathfrak{h})$.
\end{enumerate}    
\end{coro}
Suppose $M\in \mathcal{M}(\mathcal{L}, \widetilde{\mathfrak{h}}_n)$. Note that in order to understand the module structure on $M$, we need first to understand the Cartan free modules of extension of Witt algebras by finite-dimensional abelian Lie algebras.\\
Let $\mathcal{W}_n=Der(A)$ be the Witt algebra. In \cite{TZ1}, the authors defined a class of $\mathcal{W}_n$ modules as follows. For $a\in \C$ and $\Lambda_n=(\lambda_1,\lambda_2, \cdots, \lambda_n)\in (\C^*)^n$, denote $\Omega(\Lambda_n,a)=\C[d_1,d_2,\cdots, d_n]$ the polynomial algebra over $\C$ in commuting indeterminates $d_1, d_2, \cdots d_n$. The action of $\mathcal{W}_n$ on $\Omega(\Lambda_n,a)$ is defined by 
\begin{equation}\label{ma}
    t^kd_i.f(d_1,\cdots,d_n)=\Lambda_n^k(d_i-k_i(a+1))f(d_1-k_1,\cdots,d_n-k_n),
\end{equation}
 where $k=(k_1,k_2,\cdots, k_n)\in \Z^n, f\in \C[d_1,d_2,\cdots, d_n], \Lambda_n^k=\lambda_1^{k_1}\lambda_2^{k_2}\cdots \lambda_n^{k_n}, i=1,2,\cdots ,n$.
Now we recall a result from \cite{TZ}.
\begin{theo}{(Theorem 9, \cite{TZ})}
  Let $M\in \mathcal{M}(\mathcal{W}_n,D)$, where $D$ is the $\C$-linear span of $\{d_1,d_2,\dots d_n\}$. Then $M\cong \Omega(\Lambda_n,a)$ for some $\Lambda_n\in (\C^*)^n$ and some $a\in \C$.
\end{theo}
  Let $\widetilde{\mathcal{W}_n}=\mathcal{W}_n\oplus Z$ and $\widetilde{D}=D\oplus Z$, where  $Z$ is any finite-dimensional abelian Lie algebra. Now our aim is to understand the category $\mathcal{M}(\widetilde{\mathcal{W}_n}, \widetilde{D})$. Now we will follow the idea of Lemma 13 from \cite{CTZ2}. \\
 Let $M\in \mathcal{M}(\widetilde{\mathcal{W}_n}, \widetilde{D})$. By assumption $[\mathcal{W}_n, Z]=0$, so it is easy to see that $\widetilde{\mathcal{W}_n}$ module structure on $M$ is determined by the actions of $t^rd_i$ on $1$, where $r\in \Z^n$ and $1\leq i\leq n$.\\
 Define a new Lie algebra $\widehat{\mathcal{W}_n}:=\mathcal{U}(Z)\otimes_{\C} \mathcal{W}_n\subset \mathcal{U}(\widetilde{\mathcal{W}_n})$ over the PID $\mathcal{U}(Z)$. Then $M=\mathcal{U}(\widetilde{D})$ becomes a $\widehat{\mathcal{W}_n}$-module over $\mathcal{U}(Z)$ where the action of $\widehat{\mathcal{W}_n}$ on $\mathcal{U}(\widetilde{D})$ inherits from the action of $\widetilde{\mathcal{W}_n}$ on $M$.\\
 Now let us extend the base field. Suppose $\C(Z)$ be the field of fraction of $\mathcal{U}(Z)$ and $\overline{\C(Z)}$ be the algebraic closed extension field of $\C(Z)$. Let 
 \begin{equation*}
     G=\overline{\C(Z)}\otimes_{\mathcal{U}(Z)}\widehat{\mathcal{W}_n}\cong \overline{\C(Z)}\otimes_{\C} \mathcal{W}_n
 \end{equation*}
with the Cartan subalgebra $\mathfrak{D}\cong \overline{\C(Z)}\otimes_{\C} D$. Note that $\mathcal{U}(\widetilde{D})\cong \mathcal{U}(Z)\otimes_{\C}\mathcal{U}(D)$, we have vector space isomorphism over  $\overline{\C(Z)}$ as follows:
\begin{equation*}
    \overline{\C(Z)}\otimes_{\mathcal{U}(Z)}\mathcal{U}(\widetilde{D})\cong \overline{\C(Z)}\otimes_{\C}\mathcal{U}(D)\cong \overline{\C(Z)}[\mathfrak{D}]=\mathcal{U}(\mathfrak{D}),
\end{equation*}
where the last term is the universal enveloping algebra of $\mathfrak{D}$ over $ \overline{\C(Z)}$. So we can identify $\mathcal{U}(\mathfrak{D})$ with with $ \overline{\C(Z)}\otimes_{\mathcal{U}(Z)}\mathcal{U}(\widetilde{D})$. Now it is natural to extend the module action of $\widehat{\mathcal{W}_n}$ on $\mathcal{U}(\widetilde{D})$ to the module action of $G$ on  $\overline{\C(Z)}\otimes_{\mathcal{U}(Z)}\mathcal{U}(\widetilde{D})$ by
\begin{equation*}
    \phi_1\otimes \phi_2\otimes x \circ \psi_1\otimes \psi_2 g=\phi_1\psi_1\otimes(\phi_2\otimes x.\psi_2g),
\end{equation*}
where $\phi_1,\psi_1\in \overline{\C(Z)}, \; \phi_2, \psi_2 \in \mathcal{U}(Z),\; x\in \mathcal{W}_n,\; g\in \mathcal{U}(D)$.\\
We see that the action of $\mathfrak{D}$ on $\overline{\C(Z)}\otimes_{\mathcal{U}(Z)}\mathcal{U}(\widetilde{D})$ is just multiplication, so the $G$-module $\overline{\C(Z)}\otimes_{\mathcal{U}(Z)}\mathcal{U}(\widetilde{D})$ is inside $\mathcal{M}(G, \mathfrak{D})$, where all the modules are over $\overline{\C(Z)}$. \\
Now the method used in \cite{TZ} to prove Theorem 9 is valid for the $G$-module $\overline{\C(Z)}\otimes_{\mathcal{U}(Z)}\mathcal{U}(\widetilde{D})$. So $\overline{\C(Z)}\otimes_{\mathcal{U}(Z)}\mathcal{U}(\widetilde{D})\cong \Omega (\Lambda_n, a)$ for some $\Lambda_n\in ((\overline{\C(Z)})^*)^n$ and some $a\in \overline{\C(Z)}$.\\
Note that the action of $\mathcal{W}_n\subset G$ on $\mathcal{U}(\widetilde{D})$ coincide with the action of $\mathcal{W}_n$ on the $\widetilde{\mathcal{W}_n}$-module $M=\mathcal{U}(\widetilde{D})$, we must have $t^rd_i\circ 1\in \mathcal{U}(\widetilde{D})$ for $r\in \Z^n,\; 1\leq i\leq n$, which forces that $a, \lambda_1, \lambda_2,\dots, \lambda_n\in \mathcal{U}(Z)$. Again from the action \ref{ma}, we see that $\lambda_1^{-1}, \lambda_2^{-1}\dots, \lambda_n^{-1}\in \mathcal{U}(Z)$. So $\Lambda_n\in (\C^*)^n$ and $a\in \mathcal{U}(Z)$ and we have the following Proposition.
\begin{prop}\label{der}
    Let $M\in \mathcal{M}(\widetilde{\mathcal{W}_n}, \widetilde{D})$. Then $M\cong \Omega(\Lambda_n,a)$ for some $\Lambda_n\in (\C^*)^n$ and $a\in \mathcal{U}(Z)$.
\end{prop}
Suppose $M\in \mathcal{M}(\mathcal{L}, \widetilde{\mathfrak{h}}_n)$. From the Lie brackets of Full toroidal Lie algebras, we see that for any $\tilde{g}\in \mathcal{U}(\widetilde{\mathfrak{h}}_n)$: 
 \begin{equation*}
    \begin{cases}
        t^aK_j. (\tilde{g})=\tau^a (\tilde{g})(t^aK_j.1),\\
        h(a). (\tilde{g})=\tau^a (\tilde{g})(h(a).1),\\
        x_i(a).(\tilde{g})=(\sigma_i \tau^a(\tilde{g})) (x_i(a).1),\\
        y_i(a).(\tilde{g})=(\sigma_i^{-1} \tau^a(\tilde{g})) (y_i(a).1),\\
       D(u,r).(\tilde{g})=\tau^{r}(\tilde{g})D(u,r).1
    \end{cases}
     \end{equation*}
  where $1\leq i\leq l$, $j\in \overline{1,n}$, $u\in \C^n$ and $r\in \Z^n$.\\

 Note that $ {Res}^{\mathcal{L}}_{\widetilde{\mathfrak{g}_n}}M\in \mathcal{M}(\mathcal{L}, \widetilde{\mathfrak{h}}_n)$, there fore we have $t^aK_j.1=0$, where $1\leq j\leq n$. Now from Theorem \ref{main1}, we have $ {Res}^{\mathcal{L}}_{\widetilde{\mathfrak{g}_n}}M \cong \widetilde{M}_{\widetilde{\mathfrak{g}}_n}(\mathbf{\lambda},M_\mathfrak{g})$, for some $\mathbf{\lambda}\in (\C^*)^n, M_{\mathfrak{g}}\in \mathcal{M}(\mathfrak{g}, \mathfrak{h})$. \\
 
 Also $Res^{\mathcal{L}}_{\widetilde{\mathcal{W}_n}}M\in \mathcal{M}(\widetilde{\mathcal{W}_n},\widetilde{D})$, where $\widetilde{\mathcal{W}_n}=\mathcal{W}_n\oplus \mathfrak{h}_n$ and $\widetilde{D}=D\oplus \mathfrak{h}_n.$ So from Proposition \ref{der}, we see that $Res^{\mathcal{L}}_{\widetilde{\mathcal{W}_n}}M \cong \Omega (\gamma, a)$ for some $\gamma \in (\C^*)^n$ and $a\in \mathcal{U}(Z)$. \\
 {\bf Claim:} $a\in \C$.\\
We know that $[t^rd_j,x_i]=0$ for all $r\in \Z^n$, $1\leq j\leq n$ and $1\leq i\leq l$. Therefore we will have the following:
\begin{align*}
   0 &= t^rd_j.x_i.1-x_i.t^rd_j.1 \cr 
   &=\tau^r(x_i.1)t^rd_i.1-x_i.t^rd_j.1 \cr 
   & =\gamma^r(x_i.1)(d_j-r_j(a+1))-x_i.(t^rd_j.1) \cr
   &=\gamma^r\{(x_i.1)(d_j-r_j(a+1))-\sigma_i(d_j-r_j(a+1))(x_i.1)\}
    \end{align*}
   From the above equation, we see that if we choose $r\in \Z^n$ such that $r_j\neq 0$, we will have $a-\sigma_i(a)=0$ for all $1\leq i\leq l$. So $a\in \cap_{i=1}^l\mathcal{P}_i=\mathcal{P}_D$. We also know that $a\in \mathcal{U}( \mathfrak{h}_n)=\mathcal{P}_{H}$, therefore $a\in \C$. 
   Now consider the relation $[t^rd_j, x_i(s)].1=s_jx_i(r+s).1$, which will give us the following:
   \begin{align*}
       s_j\Lambda_n^{r+s}x_i.1 
       &=t^rd_j.x_i(s).1-x_i(s).t^rd_j.1 \cr
       &=\Lambda_n^st^rd_j.x_i.1-x_i(s).\gamma^r(d_j-s_j(a+1)) \cr
       &=\Lambda_n^s\tau^{r}\{(x_i.1)\}t^rd_j.1-\gamma^r\sigma_i\tau^s\{(d_j-s_j(a+1))\}x_i(s).1 \cr
       &=\Lambda_n^s(x_i.1)\gamma^r(d_j-r_j(a+1))-\Lambda_n^s\gamma^r(d_j-s_j-s_j(a+1)) \
       =s_j\Lambda_n^s\gamma^r(x_i.1).
 \end{align*}
From the above equation, if we choose $s\in \Z^n$ such that $s_j\neq 0$, then we see that $\Lambda_n^r=\gamma^r$ for all $r\in \Z^n$.  Therefore we will have $\Lambda_n=\gamma$.
\begin{theo}
    Let $M\in \mathcal{M}(\mathcal{L}, \widetilde{\mathfrak{h}}_n)$, then $M\cong \widehat{M}_{\mathcal{L}}(\Lambda,a,M_{\mathfrak{g}})$, for some $\Lambda \in (\C^*)^n,\; a\in \C,\; M_{\mathfrak{g}}\in \mathcal{M}(\mathfrak{g},\mathfrak{h})$.
\end{theo}
 \begin{proof}
     The proof follows from the above discussion.
 \end{proof}  
\begin{theo}
    $\widehat{M}_{\mathcal{L}}(\Lambda,a,M_{\mathfrak{g}}) \cong \widehat{M}_{\mathcal{L}}(\Lambda',a',M_{\mathfrak{g'}})$ if and only if $M_{\mathfrak{g}} \cong$ $ M_{\mathfrak{g'}}$ as $\mathfrak g$-module and $\Lambda=\Lambda'$, $a=a'.$
\end{theo} 
\begin{proof}
   It is clear to see the sufficient part.  Let $\phi:\widehat{M}_{\mathcal{L}}(\Lambda,a,M_{\mathfrak{g}}) \to \widehat{M}_{\mathcal{L}}(\Lambda',a',M_{\mathfrak{g'}}) $ be the isomorphism with its inverse $\phi^{-1}$. Therefore we have the following relations
    \begin{align}
			\phi(\tilde g)=\tilde g\phi(1) \\
			\phi^{-1}(\tilde g)=\tilde g\phi^{-1}(1),
		\end{align}
  for all $\tilde g \in U(\tilde{\mathfrak h}_n).$ In particular we have $\phi^{-1}(\phi(1))=\phi(1)\phi^{-1}(1)=1$, which implies that $\phi(1) $ is a non-zero scalar. Now consider $\psi=\phi|_{U(\mathfrak h)}$ and note that $\psi(U(\mathfrak h)=U(\mathfrak h).$ Hence $\psi$ is an isomorphism between $M_{\mathfrak{g}}$ and $ M_{\mathfrak{g'}}$.\\
   Let $\Lambda=(\lambda_1,\dots , \lambda_n)$ and $\Lambda'=(\lambda_1', \dots, \lambda_n')  .$ Now consider the relation $\phi(h(e_i).\tilde g)=h(e_i).\phi(\tilde g)=h(e_i).\tilde g\phi(1)$. This implies that $\lambda_i=\lambda_i'.$ Also consider the relation $\phi(t_id_i.\tilde g)=t_id_i.\phi(\tilde g)=t_id_i.\tilde g\phi(1)$ and conclude that $a=a'$ using the fact that $\lambda_i=\lambda_i'.$
\end{proof}


\begin{thebibliography}{100}
\bibitem{BBFK} V. Bekkert, G. Benkart, V. Futorny, I. Kashuba, {\sl New irreducible modules for Heisenberg and affine
Lie algebras}, J. Algebra {\bf 373} (2013) 284–298.
 \bibitem{BB} S. Berman, Y. Billig, {\sl Irreducible representations for toroidal Lie algebras}, J. Algebra {\bf 221} (1999) 188-231.
 \bibitem{BGG} I. N. Bernstein, I. M. Gelfand, S. I. Gelfand, {\sl A certain category of $\mathfrak{g}$-modules} Funkcional. Anal. i Prilozen {\bf 10} (1976) 1-8. 
 \bibitem{RB} R. Block, {\sl The irreducible representations of the Lie algebra $\mathfrak{sl}(2)$ and of the Weyl algebra}, Adv. Math. {\bf 139. (1)} (1981) 69-110. 
 \bibitem{BP} R. Blumenhagen, E. Plauschinn, {\sl Introduction to Conformal Field Theory: With Applications to
String Theory}, Lect. Notes Phys., {\bf vol. 779}, Springer, Berlin Heidelberg, 2009.
\bibitem{CH} Q. Chen, J. Han, {\sl Non-weight modules over the affine-Virasoro
algebra of type $A_1$,} J. Math. Phys. {\bf 60}, 071707 (2019).
\bibitem{CTZ1} Y. A. Cai, H. Tan, K. Zhao, {\sl New representations of affine Kac-Moody algebras}, J. Algebra {\bf 547} (2020), 95-115.
\bibitem{CTZ2} Y. A. Cai, H. Tan, K. Zhao, {\sl Module structures on $U(\mathfrak{h})$ for Kac-Moody algebras,} Sci. Sin. Math. {\bf 47 (11)} (2017) 1491-1514, arXiv:1606.01891.
\bibitem{DG} I. Dimitrov, D. Grantcharov, {\sl Classification of simple weight modules over affine Lie algebras}, arXiv:
0910.0688v1.
\bibitem{EC} E. Cartan, {\sl Les groups projectifs qui ne laissent invariante aucune multiplicite` planet,} Bull. Soc. Math. France {\bf 41} (1913) 53-96.
\bibitem{VC} V. Chari, {\sl Integrable representations of affine Lie-algebras}, Invent. Math. {\bf 85} (1986) 317–335.
\bibitem{CP} V. Chari, A. Pressley, {\sl Integrable representations of twisted affine Kac-Moody algebras}, J. Algebra
{\bf 113} (1988) 438–464.
\bibitem{CP1} V. Chari, A. Pressley, {\sl Integrable modules for affine Lie algebras}, Math.Ann. {\bf 277 (3)} (1987) 543–562.
\bibitem{CP2} P. Chakraborty,  {\sl Irreducible modules for map Heisenberg-Virasoro Lie algebras}, arXiv:2311.02635 .

\bibitem{DOF} Yu. Drozd, S. Ovsienko, V. Futorny, {\sl On Gelfand–Zetlin modules}, in: Proceedings of the Winter School on Geometry and Physics, Srni, 1990, in: Rend. Circ. Mat. Palermo {\bf (2), vol. 26}, 1991,
pp. 143–147.
\bibitem{JD} J. Dixmier, Enveloping Algebras, American Mathematical Society, 1977. 
\bibitem{VF} V. Futorny, {\sl Irreducible non-dense $A(1)^1$ -modules}, Pacific J. Math. {\bf 172} (1996) 83–99.
\bibitem{VF1} V. Futorny, {\sl Representations of Affine Lie Algebras}, Queens Papers in Pure and Appl. Math., {\bf vol. 106}, Queens University, Kingston, ON, 1997.
\bibitem{FT} V. Futorny, A. Tsylke,{\sl Classification of irreducible nonzero level modules with finite-dimensional
weight spaces for affine Lie algebras}, J. Algebra {\bf 238} (2001) 426–441.
\bibitem{FK} V. Futorny, I. Kashuba, {\sl Structure of parabolically induced modules for affine Kac-Moody algebras},
J. Algebra {\bf 500} (2018) 362–374.
\bibitem{OM} O. Mathieu, {\sl Classification of irreducible weight modules,} Ann. Inst. Fourier (Grenoble) {\bf 50 (2)} (2000) 537-592.
\bibitem{VM} V. Mazorchuk, Lectures on $\mathfrak{sl}_2(\C)$-Modules, Imperial College Press, London, 2010.
\bibitem{JH} J. E. Humphreys, {\sl Representations of Semisimple Lie algebras in BGG category $\mathcal{O}$,} American Mathematical Society, 2008.
\bibitem{JK} H.P. Jakobsen, V.G. Kac, {\sl A new class of unitarizable highest weight representations of infi-
nite dimensional Lie algebras}, in: Nonlinear Equations in Classical and Quantum Field Theory,
Meudon/Paris, 1983/1984, in: Lecture Notes in Phys., {\bf vol. 226}, Springer, Berlin, 1985, p. 120.
\bibitem{BK}  B. Kostant, {\sl  On Whittaker vectors and representation theory}, Invent. Math. {\bf 48 (2)} (1978) 101–184.
\bibitem{VK} V.G. Kac, {\sl Infinite-dimensional Lie algebras and Dedekind’s $\eta$-function,} Funkt. Anal. Prilozh {\bf 8 (1)} (1974) 77–78, English translation Funct. Anal. Appl. {\bf 8} (1974) 68–70.
\bibitem{VK1} V.G. Kac, {\sl Infinite Dimensional Lie Algebras}, Cambridge University Press, 1990.
\bibitem{KR} V.G. Kac, A. Raina, {\sl Bombay lectures on highest weight representations of infinite dimensional Lie
algebras}, World Sci, Singapore, 1987.
 \bibitem{LL} J. Lepowsky, H. Li, {\sl Introduction to Vertex Operator Algebras and Their Representations, Progress
in Mathematics}, {\bf vol. 227}, Birkhauser Boston Inc., Boston, 2004. 
\bibitem{MZ}  V. Mazorchuk, K. Zhao, {\sl Characterization of simple highest weight modules}, Canad. Math. Bull.
{\bf 56 (3) (2013)} 606–614.
\bibitem{N1} J. Nilsson, {\sl Simple $\mathfrak{sl}_{n+1}$-module structures on $U(h)$}, J. Algebra {\bf 424} (2015) 294-329.
\bibitem{N2} J. Nilsson, {\sl $\mathcal{U}(\mathfrak{h})$-free modules and coherent families}, J. Pure Appl. Algebra {\bf 220(4)} (2016) 1475-1488.
\bibitem{ER}  S.E. Rao, {\sl Classification of irreducible integrable modules for toroidal Lie algebras with finite-dimensional weight spaces }, J. Algebra 277 (2004), no. 1, 318–348.
\bibitem{RJ} E. S. Rao, C. Jiang, {\sl Classification of irreducible integrable representations for the full toroidal Lie algebras}, J. Pure Appl Algebra {\bf 2005} (2005) 71-85.
\bibitem{RMY} E. S. Rao, R. V. Moody, T. Yokonuma, {\sl Toroidal Lie algebras and vertex representations}, Geom, Dedicata {\bf 35} (1990), 283-307.
\bibitem{K} C. Kassel, {\sl Kahler differentials and coverings extended over a commutative algebra}, Journal of Pure and Applied Algebra {\bf 34} (1984), 265-275. 
\bibitem{TZ}H. Tan, K. Zhao, {\sl $\mathcal{W}_n^+$- and $\mathcal{W}_n$-module structures on $\mathcal{U} (\mathfrak{h}_n)$}, J. Algebra {\bf 424} (2015) 357-375. 
\bibitem{TZ1}H. Tan, K. Zhao, {\sl Irreducible modules over Witt algebras $\mathcal{W}_n$ and over $\mathfrak{sl}_{n+1}(\C)$}, Algebr. Represent. Theory {\bf 21} (2018), no. 4, 787–806.
\bibitem{JZ} J. Zhang, {\sl Non-weight representations of Cartan type S Lie algebras}, Comm. Algebra 2018, Vol. 46, NO. 10, 4243-4264.
\bibitem{XZ} X. Zhu, {\sl Simple modules over Takif Lie algebra $\mathfrak{sl}_2$}, arXiv:2211.07261v1.
\end{thebibliography}
\end{document}